\documentclass[11pt]{elsarticle}
\usepackage{geometry}
\usepackage{sectsty}

\usepackage{amsmath}
\usepackage{amssymb}
\usepackage{mathtools} %for ceiling function
\usepackage{amsthm}
\usepackage{thmtools}
\usepackage[hidelinks]{hyperref}
\usepackage[nameinlink]{cleveref}

\usepackage{color}
\usepackage{graphicx}

\newcommand{\reals}{\ensuremath{\mathbb{R}}}
\newcommand{\norm}[1]{\ensuremath{\left\vert\left\vert#1\right\vert\right\vert}}
\newcommand{\pdiff}[1]{\ensuremath{\frac{\partial}{\partial #1}}}

\newcommand{\Pdiff}[2]{\ensuremath{\frac{\partial #1}{\partial #2}}}
\newcommand{\Pdifftwo}[2]{\ensuremath{\frac{\partial^2 #1}{\partial {#2}^2}}}
\newcommand{\Pdifftwov}[3]{\ensuremath{\frac{\partial^2 #1}{\partial {#2}\partial{#3}}}}

\newcommand{\idiff}[2]{\ensuremath{\partial #1/\partial #2}}

\newcommand{\cdiff}[1]{\ensuremath{\frac{D}{\partial #1}}}

\newcommand{\vbar}{\ensuremath{\biggr\vert}}
\newcommand{\flow}{\ensuremath{\mathcal{F}}}

\newcommand{\T}[1]{\ensuremath{T_\partial #1}}
\newcommand{\uT}[1]{\ensuremath{T_1 #1}}
\newcommand{\genflow}{\ensuremath{\mathfrak F}}

\newcommand{\pr}{\ensuremath{\textrm{pr}}}
\newcommand{\grad}{\ensuremath{\textrm{grad }}}

\newcommand{\sT}{\ensuremath{\mathcal{T}}}
\newcommand{\TS}{\ensuremath{T^+\partial S}}
\newcommand{\Trap}[1]{\ensuremath{\textrm{Trap}(#1)}}

\newcommand{\mindist}{\ensuremath{d_{\min}}}
\newcommand{\maxdist}{\ensuremath{D}}
\newcommand{\maxsec}{\ensuremath{sec_{\max}}}
\newcommand{\mincurv}{\ensuremath{\kappa_{\min}}}
\newcommand{\minref}{\ensuremath{\xi}}
\newcommand{\globalcurv}{\ensuremath{\Theta}}
\newcommand{\frontmincurv}{\ensuremath{k_{\min}}}

\newcommand{\comment}[1]{}

\newtheorem{theorem}{Theorem}
\newtheorem{lemma}[theorem]{Lemma}
\newtheorem{proposition}[theorem]{Proposition}
\newtheorem{corollary}{Corollary}[theorem]

\theoremstyle{definition}

\crefname{condition}{condition}{conditions}

\date{July 29, 2023}

\begin{document}
\begin{frontmatter}
\title{Uniqueness of Obstacles in Riemannian Manifolds from Travelling Times}
\author{Tal Gurfinkel\corref{cor1}}\ead{tal.gurfinkel@research.uwa.edu.au}
\author{Lyle Noakes}\ead{lyle.noakes@uwa.edu.au}
\author{Luchezar Stoyanov}\ead{luchezar.stoyanov@uwa.edu.au}
\cortext[cor1]{Corresponding Author}
\address{Department of Mathematics and Statistics, University of Western Australia, Crawley 6009 WA, Australia}
\tnotetext[t1]{This research is supported by an Australian Government Research Training Program (RTP)
 Scholarship.}
\begin{abstract}
	Suppose that $K$ and $L$ are two disjoint unions of strictly convex obstacles with the same set of travelling times, contained in an $n$-dimensional Riemannian manifold $M$ (where $n\geq2$). Under some natural curvature conditions on $M$, and provided that no geodesic intersects more than two components in $K$ or $L$, we show that $K = L$.
\end{abstract}
\begin{keyword}
	Travelling Times, Inverse Scattering, Convex Obstacles, Riemannian Manifolds.
	\MSC[2020] 37D40, 37C83, 53C21
\end{keyword}
\end{frontmatter}
Boundary rigidity problems have been important in Riemannian geometry for over one hundred years, thanks to their practical applications in tomography and geophysics \cite{MR3952705}. Inverse problems similar to boundary rigidity problems, such as lens rigidity problems and marked length spectrum problems have also been studied \cite{MR3600043,MR3743701}. We also consider an inverse problem related to boundary rigidity, concerning the recovery of the boundary of many disjoint obstacles contained within an exterior body, from scattering data known only along the boundary of the exterior body. Classically boundary rigidity problems make use of data given by the geodesic flow, while our data will be determined by the billiard flow. Note that we aim to determine the boundary of several bodies while in boundary rigidity problems this information is already known and the objective is to recover the metric on the given manifold.

We suppose that $M$ is a complete $n$-dimensional Riemannian manifold ($n\geq 2$). We denote the metric on $M$ by $\langle\ \cdot\ \rangle$, and the sectional curvature of $M$ by $sec_M(X,Y)$ for any $X,Y\in TM$. We say that an $n$-dimensional submanifold $W$ of $M$ is strictly convex if the second fundamental form of $\partial W$ is positive definite with respect to the inward normal field of $\partial W$. We ought to remark that this notion of convexity is sometimes known as infinitesimal convexity. It is known that strict convexity implies local convexity \cite{MR350662}. Furthermore, under certain curvature conditions on $M$, it is known that if $H$ is a strictly convex hypersurface of $M$ then $H$ is the boundary of a convex body $\widetilde H$ in $M$. That is, $\widetilde H$ is convex in the sense that any two points in $\widetilde H$ are connected by a unique geodesic of $M$ which is contained in $\widetilde H$ (cf. \cite{MR448262,MR869706,MR962285}). Nonetheless for our purposes it is sufficient to only assume that strict convexity of the boundary holds. Let $S$ be a $n$-dimensional, strictly convex submanifold of $M$ with smooth boundary, such that between any two points $p,q\in S$ there is a unique minimal smooth geodesic of $S$ connecting $p$ and $q$. Note that this does not necessarily imply that for every $p,q\in S$ there is a unique minimal smooth geodesic of $M$ between $p$ and $q$ which is contained in $S$.
By an \emph{obstacle} we mean a union $K = K_1 \cup \dots\cup K_d$ of $d\geq 2$ disjoint $n$-dimensional, strictly convex submanifolds $K_i$ of $S$ with smooth boundary. Let $S_K = \overline{S\backslash K}$, we denote the (smooth) geodesic flow induced by $M$ as $\flow$ and the billiard flow induced by $S_K$ as $\genflow^K$. Denote the unit tangent bundle of $S_K$ by $\uT{S_K}$, and for any $\sigma\in\uT{S_K}$, let $\gamma_\sigma(t) = \pr_1\circ\genflow^K_t(\sigma)$ be the billiard ray generated by $\sigma$, where $\pr_1:TM\to M$ is the usual projection. If there are finite distinct times $\underline t \leq 0 \leq \overline t$ such that $\gamma_\sigma(\underline t),\gamma_\sigma(\overline t)\in\partial S$ we say that $\gamma_\sigma$ is a \emph{non-trapped} ray. Otherwise we say that it is \emph{trapped}, and denote the set of $\sigma\in\uT{S_K}$ such that $\gamma_\sigma$ is trapped by $\Trap{S_K}$. Denote the inward normal of $\partial S$ by $N_S$, and the set of inward pointing unit vectors on $\partial S$ by
\[
	\TS = \{\sigma\in \uT{S} : \pr_1(\sigma)\in\partial S \textrm{ and } \langle \sigma, N_S \rangle > 0\},
\]
Denote the set of trapped rays in $\TS$ by $\Trap{\partial S}^{(K)} = \Trap{S_K}\cap\TS$. For any non-trapped $\sigma\in\TS\backslash\Trap{\partial S}^{(K)}$, there is a time $t(\sigma)>0$ such that $y(\sigma) = \gamma_\sigma(t(\sigma))\in\partial S$. Let $x(\sigma) = \gamma_\sigma(0)$. The set of \emph{travelling times} of $K$ is defined as
\[
	\sT_K = \{(x(\sigma),y(\sigma),t(\sigma)):\sigma\in\TS\backslash\Trap{\partial S}^{(K)}\}.
\]
Suppose that $K$ and $L$ are two obstacles which have the same set of travelling times. That is $\sT_K = \sT_L$. When $M$ is $\reals^n$, it is known (see \cite{MR3359579}) that $K = L$. Furthermore, when $n=2$, and the conditions in \cref{condition:neg_curve} or \cref{condition:pos_curve} (see below) hold\footnote{Note that in this case $M$ does not have to be $\reals^2$, any Riemannian manifold satisfying the conditions will suffice. Setting $n=2$ means that $\minref = 1$ and $\varphi_0 = 0$ are sufficient.} then one can reconstruct $K$ directly from the set of travelling times (see \cite{math9192434,constructivealgorithm}), provided $K$ is in \emph{general position}, that is, no smooth geodesic in $S$ will intersect more than 2 distinct components of $K$.

Denote the minimum distance between any two distinct components in $K$ by $\mindist^K$ and in $L$ by $\mindist^L$. Set $\mindist = \min\{\mindist^K,\mindist^L\}$, and let $D$ be the diameter of $S$. Since $S$ is compact, there is a constant $\maxsec\in\reals$ such that $sec_M(X,Y) \leq \maxsec$ for all $X,Y\in\uT{S}$. Since $K$ and $L$ are disjoint unions of strictly convex submanifolds, there exists a lower bound $\mincurv > 0$ on the principal curvatures of $\partial K$ or $\partial L$. Let $\minref = \lceil{\frac{\maxdist}{\mindist}}\rceil + 2$. By \Cref{lemma:minimal_reflection_angle} there are angles $\varphi_0^K,\varphi_0^L\in(0,\pi/2)$ such that every ray in $S_K$ and $S_L$ respectively will hit $\partial K$ and $\partial L$ at an angle of at most $\varphi_0^K$ and $\varphi_0^L$ with respect to the outward normals of $\partial K$ and $\partial L$ at least once very $\minref$ reflections. We set $\varphi_0 = \min\{\varphi_0^K,\varphi_0^L\}$. In \Cref{proposition:convex_front_from_obstacle} we construct strictly convex fronts from the tangent rays of the obstacles. Such a front is shown to have a lower bound $\Theta_0>0$ on its minimum principal curvature. We note that $\varphi_0$ and $\Theta_0$ both depend only on $K$ and $L$.
We ask that $S$ satisfies one of the following conditions:
\begin{equation}\label{condition:neg_curve}
	\maxsec \leq 0\textrm{ or,}
\end{equation}
\begin{equation}\label{condition:pos_curve}
	\maxsec > 0,\textrm{ while }\maxdist\minref\sqrt{\maxsec}<\frac{\pi}{2}\textrm{ and }\tan(\maxsec\maxdist\minref)\sqrt{\maxsec}<\globalcurv.
\end{equation}
\[\globalcurv = \min\{2\mincurv\cos\varphi_0,\globalcurv_0\}.\]
 
Given any $\widetilde\sigma\in T_1S\backslash\Trap{S_K}$ which generates a trajectory $\gamma^{K}_{\widetilde\sigma}$ in $S_K$ that is tangential to $\partial K$ at some point. There is some $\sigma\in\TS$ such that $\gamma^K_\sigma$ and $\gamma^K_{\widetilde\sigma}$ are the same ray (up to re-parameterisation). Let $t_K^*>0$ be the minimum time at which $\gamma^{K}_\sigma$ is tangent to $K$. Since $\sT_K = \sT_L$ it follows that $\gamma^L_\sigma$ must be tangent to $\partial L$ as well. This holds owing to the fact that $\sigma$ is a singularity of the travelling time function $t_K:\TS\backslash\Trap{\partial S}^{(K)}\to\sT_K$. Since $t_L = t_K$, the point $\sigma$ must also be a singularity of $t_L$ and hence generate a tangent ray in $S_L$. Denote the minimum time at which $\gamma^{L}_\sigma$ is tangent to $\partial L$ by $t_L^*>0$. Set $t^* = \min\{t^*_K,t^*_L\}$.
We say that $K$ and $L$ are \emph{equivalent up to tangency} if for all such $\sigma\in T_1S$,
\[
	\gamma_\sigma^K(t) = \gamma_\sigma^L(t)\quad \textrm{for all}\quad 0\leq t \leq t^*.
\]

%For any $\sigma\in\uT{S_K}\cap \uT{S_L}$ let $\gamma_\sigma^K(t) = \pr_1\circ\genflow^K_t(\sigma)$, and $\gamma_\sigma^L(t) = \pr_1\circ\genflow^L_t(\sigma)$.
%We say that $K$ and $L$ are \emph{tangentially equivalent} if whenever $\gamma_\sigma^K$ is tangent to $\partial L$, the ray $\gamma_\sigma^L(t)$ is also tangent to $\partial L$ and vice versa. 
This allows us to state our main theorem:

\begin{theorem}\label{Theorem 1}
	Let $K$ and $L$ be two disjoint unions of strictly convex obstacles contained within the same strictly convex submanifold $S$ of $M$. Suppose that $K$ and $L$ have the same set of travelling times, and satisfy either \cref{condition:neg_curve} or \cref{condition:pos_curve}. If $K$ and $L$ are equivalent up to tangency, then $K = L$. Consequently if $K$ and $L$ are both in general position then $K$ = $L$.
\end{theorem}

Denote by $\mathcal{K}_T$ the class of pairs $(K,L)$ of obstacles  with the same travelling times that are equivalent up to tangency. We remark that this class is nonempty, in fact if a pair of obstacles $(K,L)$ both satisfy Ikawa's no-eclipse condition (see \cite{MR949013}) and have the same set of travelling times then they are equivalent up to tangency. i.e. $(K,L)\in\mathcal{K}_T$. To show that this is in fact true, recall Ikawa's no-eclipse condition is as follows: for any three distinct components $K_i,K_j,K_l\subseteq K$, if $\textrm{Hull}(K_i\cup K_j)$ is the convex hull of $K_i\cup K_j$ then $\textrm{Hull}(K_i\cup K_j)\cap K_l = \emptyset$. We note that this is equivalent to saying that $K$ is in general position. \Cref{lemma:gen_pos_equiv} shows that if $K$ and $L$ are both in general position, respectively, then they must be equivalent up to tangency.

\section{The Propagation of Convex Fronts}\label{sect:propagation}

Let $X$ be a strictly convex codimension 1 submanifold of $M$. Let $N$ be a unit normal field to $X$, then the second fundamental form of $X$ is \[h(Y,Z) = \langle\nabla_{Y}Z, N\rangle = -\langle \nabla_Y N, Z\rangle,\]
for any vectors $Y,Z$ tangent to $X$. Then $h$ also defines the \emph{shape operator} of $X$, a linear self-adjoint map $s$, where $\langle -sY, Z \rangle = h(Y,Z)$. Also note that $sY = \nabla_Y N$.
Now define the submanifold $X_t$ as 
\[X_t = \{\gamma_{(x,N(x))}(t): x\in X\}.\]
Here $\gamma_{(x,N(x))}$ is the geodesic in $M$ starting from $x\in X$ in the direction of the outward normal $N(x)$ to $X$. We may consider $h_t$ along each geodesic normal to $X$, to obtain a relation between the second fundamental forms $h_0$ and $h_t$ of $X$ and $X_t$ respectively. When we do so, we take the parallel translates of vectors $Y,Z$ from $X$ to $X_t$ along the geodesic. The purpose of this section is to describe, as directly as possible, the evolution of the curvature of $X_t$ as it is propagated forward via the billiard flow. We do this either in terms of the operator $s$, or the principal curvatures of $X_t$.
\begin{proposition}
	Suppose that $\gamma_{x,N(x)}$ does not undergo any reflections for all $x\in X$, and all $t_0<t<t_1$.
	Then for all $t_0<t<t_1$, the shape operator $s$ of $X_t$ satisfies the following differential equation,
	\begin{equation}
		\dot s Y = R(N,Y)N - s^2 Y,
	\end{equation}
	for all $Y$ tangent to $X$. Moreover, if $k$ is a principal curvature of $X_t$ then it satisfies the following differential equation:
	\begin{equation}\label{eq:principal_curvature_diff}
		\dot k = -sec_M(N,V) - k^2,
	\end{equation}
	where $V$ is the principal eigenvector corresponding to $k$.
\end{proposition}
\begin{proof}
	We compute the derivative:
\begin{align*}
	\dot s Y &= \nabla_N\nabla_Y N \\
	&= R(N,Y)N + \nabla_Y\nabla_N N + \nabla_{[N,Y]}N.
\end{align*}
Consider the term $[N,Y]$, by the symmetry of the connection we have:
\begin{align*}
	[N,Y] = \nabla_N Y - \nabla_Y N = -\nabla_Y N = -sY.
\end{align*}
Now since $\nabla_N N = 0$ we have,
\begin{align*}
	\dot s Y &= R(N,Y)N - \nabla_{sY} N\\
	&= R(N,Y)N - s^2 Y.
\end{align*}
Note that the eigenvalues of $s$ are the principal curvatures, so given an eigenvector $V$ of $s$ such that $sV=kV$ and $\langle V,V\rangle = 1$, we have $k = \langle sV, V\rangle$. Thus,
\begin{align*}
	\dot k &= \langle \dot s V, V\rangle\\
	&= \langle R(N,V)N - s^2 V, V\rangle\\
	&= -sec_M(N,V) - k^2.
\end{align*}
Recall that by $sec_M(N,V)$, we denote the sectional curvature of $M$ with respect to $N$ and $V$.
\end{proof}

\begin{corollary}\label{lemma:always_convex}
	Suppose that $\globalcurv$ is a lower bound for the principal curvatures of $X$. Then there exists a global lower bound $\frontmincurv>0$ for the principal curvatures of $X_t$ provided that $t < \maxdist\minref$.
\end{corollary}
For a proof, see Theorem 3.1 in \cite{MR807598}. This involves solving a differential equation:
	\[
		\dot k = -\maxsec - k^2,\quad k(0) = \globalcurv,
	\]
which bounds \cref{eq:principal_curvature_diff} below, for all time $0<t<\maxdist\minref$. Note that for \Cref{lemma:always_convex} to be true, the proof depends directly on the bounds given in \cref{condition:neg_curve} and \cref{condition:pos_curve}.

\begin{proposition}\label{proposition:reflection_sff}
	Suppose that $\gamma_{x,N(x)}$ reflects off an obstacle $K$ at $t_r$. Let $s^-$ and $s^+$ be the shape operators of $X$ before and after reflection respectively, at the point of reflection $\gamma_{x,N(x)}(t_r)\in K$. Let $s_K$ be the shape operator of $\partial K$. Then shape operators satisfy the following equation,
	\begin{equation}\label{eq:reflection_sff}
		s^+(Y_+) - s^-(Y_-) = -2\langle N_-, N_K\rangle s_K(Y),
	\end{equation}
	for all $Y$ tangent to $K$. Here we denote by $Y_- = Y - \langle Y, N_- \rangle N_-$ and similarly for $Y_+$, where $N_-$ is the normal to $X$ prior to reflection and $N_K$ is the normal to $K$.
\end{proposition}
\begin{proof}
	We note that along $K$, we can define the vector field $N_-$ of outward normals to $X$ prior to reflection. We then define $N_+$ as the vector field along $K$ such that \[N_+ - N_- = -2\langle N_-, N_K\rangle N_K.\]
	Now we may take covariant derivatives along $K$, in any tangent direction $Y$,
	\[
		\nabla_Y N_+ - \nabla_Y N_- = -2Y\langle N_-, N_K\rangle N_K - 2\langle N_-, N_K\rangle\nabla_Y N_K.
	\]
	Then taking inner products with respect to any tangent vector $Z$ to $K$, we have
	\[
		\langle\nabla_Y N_+, Z\rangle - \langle\nabla_Y N_-, Z\rangle = - 2\langle N_-, N_K\rangle\langle\nabla_Y N_K, Z\rangle.
	\]
	Denote the projections of $Y$ on to the tangent space of $X$ before and after reflection by $Y_-$ and $Y_+$ respectively. i.e.
	\[
		Y_\pm = Y - \langle Y, N_\pm \rangle N_\pm.
	\] 
	We similarly write $Z_\pm$ for the projections of $Z$. Thus we have
	\begin{equation}\label{eq:reflection_sff_bilinear}
		\langle\nabla_{Y_+} N_+, Z_+\rangle - \langle\nabla_{Y_-} N_-, Z_-\rangle = - 2\langle N_-, N_K\rangle\langle\nabla_Y N_K, Z\rangle.
	\end{equation}
	Which immediately implies \cref{eq:reflection_sff}.
\end{proof}
\begin{corollary}\label{lemma:reflection_curvature}
	If $\frontmincurv$ is the minimum principal curvature of $X$ prior to reflection on $\partial K$ at time $t_r$, and $k_+$ is any principal curvature after reflection, then
	\begin{equation}
		k_+ \geq \frontmincurv + 2\mincurv\cos\varphi,
	\end{equation}
	where $\varphi\in (0,\frac{\pi}{2})$ is the angle between the outward normal to $\partial K$ and the normal to $X$ after reflection, and $\mincurv$ is the minimal principal curvature of $\partial K$.
\end{corollary}
\begin{proof}
	Suppose $V_+$ is a principal direction with $\norm{V_+} = 1$, and let $k_+$ be the corresponding principal curvature,. We may pick $Y$ such that $Y_+ = V_+$, using the notation for $Y_+$ from the proof of \cref{proposition:reflection_sff}. That is, we set
	\[ 
		Y = V_+ - \frac{\langle N_K,V_+ \rangle}{\langle N_K,N_+ \rangle}N_+.
	\]
	Note that $\norm{Y_-}=1$ by construction since $Y_- = Y_+ - 2\langle Y_+, N_K\rangle N_K$. Although $Y_-$ might not be a principal direction of $X$ prior to reflection. 
	By setting $Z_\pm = Y_\pm$ it now follows from \cref{eq:reflection_sff_bilinear} that,
	\[
		k_+ \geq \frontmincurv + 2\mincurv\cos\varphi \norm{Y}^{2}.
	\]
	Then we note that
	\[
		\norm{Y}^2 = 1+\frac{\langle V_+, N_K \rangle^2}{\langle N_+,N_K \rangle^2}.
	\]
	The result follows.
\end{proof}
\begin{lemma}\label{lemma:minimal_reflection_angle}
		There exist constants $\minref\in\mathbb{Z}^+$ and $\varphi_0\in (0,\frac{\pi}{2})$ such that any geodesic reflecting transversly on $\partial K$ at least $\minref$ times will hit $\partial K$ at an angle $\varphi < \varphi_0$ at least once, with respect to the outward normal.
\end{lemma}
\begin{proof}
	Suppose the contrary. Then there exists a sequence $\{\sigma_i\}_{i=1}^\infty\subseteq \uT{S_K}$ such that the billiard ray $\gamma_{\sigma_i}$ generated by $\sigma_i$ has $i$ reflections in $S_K$, and at each point of reflection the angle between $\dot\gamma_{\sigma_i}$ and the outward normal to $\partial K$ is greater than $\pi/2 - 1/i$. Since $\uT{S_K}$ is compact it follows that there is a convergent subsequence $\{\sigma_{i_j}\}_{j=1}^\infty\to\sigma^*\in\uT{S_K}$. Then by construction the billiard ray $\gamma_{\sigma^*}$ has infinitely many points of reflection, all tangential to $\partial K$. That is, $\gamma_{\sigma^*}$ is a smooth geodesic of infinite length in $S_K$. This is a contradiction, since the maximal length of a smooth geodesic in $S$ is $D$.
\end{proof}
The following \namecref{proposition:convex_front_from_obstacle} allows one to encode the tangent ray to a small neighbourhood in $\partial K$ as a strictly convex front. The proof is rather long and technical, however the result is essential for several results which follow.
\begin{proposition}\label{proposition:convex_front_from_obstacle}
	Fix a point $x_0\in \partial K$ and tangent direction $V\in T_{x_0}\partial K$ such that $\norm{V} = 1$. There is a neighbourhood $U\subseteq \partial K$ of $x_0$ and a strictly convex front $Y$ such that for every $y\in Y$, the ray in the inward normal direction from $y$ will intersect $U$ tangentially. Hence, the front $Y$ is diffeomorphic to $U$. Furthermore, there is a global lower bound $\globalcurv_0>0$ such that the minimum principal curvature of $Y$ is greater than $\globalcurv_0$ for all $x_0\in \partial K$ and $V\in T_{x_0}\partial K$. 
\end{proposition}
\begin{proof}
	Pick a small $\varepsilon>0$, and let $\exp_{x_0}^{(\partial K)}$ be the exponential map on $\partial K$ at $x_0$. Define $x_0^* = \exp_{x_0}^{(\partial K)}(-\varepsilon V)$. It is known \cite{MR618545} that the geodesic sphere \[\partial B_{\varepsilon}(x_0^*) = \{x\in\partial K:d^{(\partial K)}(x,x_0^*)=\varepsilon\},\] has a strictly positive second fundamental form, provided $\varepsilon$ is sufficiently small. In fact, given any $K_B>0$, there exists a $\varepsilon>0$ such that the minimum principal curvature of $\partial B_\varepsilon(x_0^*)$ is bounded below by $K_B$. Let $0<\delta<\varepsilon$ and denote the family of spheres given by $\partial B_{\varepsilon^*}(x_0^*)$, where $\varepsilon^* \in (\varepsilon-\delta ,\varepsilon+\delta)$. Note that $x_0\in \partial B_{\varepsilon}(x_0^*)$ by construction. Let $U^*\subseteq \partial B_{\varepsilon}(x_0^*)$ be a neighbourhood of $x_0$. Parameterise $U^*$ as $\widetilde x(u_2,\dots,u_{n-1})$ and let $\eta(u_2,\dots,u_{n-1})$ be the outward unit normal field to $\partial B_\varepsilon(x_0^*)$. Denote the domain of $\widetilde x$ by $U_0^*$. Now for each fixed $(u_2,\dots,u_{n-1})$ let $x(u_1,u_2,\dots,u_{n-1})$ be the unit-speed geodesic in $\partial K$ from $\widetilde x(u_2,\dots,u_{n-1})$ in the direction $\eta(u_2,\dots,u_{n-1})$. Then $x(u_1,u_2,\dots,u_n)\in \partial B_{(\varepsilon + u_1)}(x_0^*)$ for all $u_1 \in (-\delta,\delta)$. Therefore, provided $\delta>0$ is sufficiently small, $x:(-\delta,\delta)\times U_0^*\to U$ is a parameterisation of a neighbourhood $U\subseteq\partial K$ of $x_0$ such that $U^*\subset U$ and $\idiff{x}{u_1}(0) = V$.
	
%	Let $X'$ be a strictly convex codimension 1 submanifold of $X$ such that $X'$ is normal to $V$ at $x_0$. Parameterise $X'$ around $x_0$ as $\widetilde x(u_2,\dots,u_{n-1})$, such that $\norm{\idiff{\widetilde x}{u_i}} = 1$ (for $i=2,\dots,n-1$). Let $\eta(u_2,\dots,u_{n-1})$ be a unit normal field to $X'$ which is tangential to $X$ (that is, $\eta$ extends $V$ along $X'$). Now for each fixed $(u_2,\dots,u_{n-1})$ let $x(u_1,u_2,\dots,u_{n-1})$ be the unit-speed geodesic in $X$ from $\widetilde x(u_2,\dots,u_{n-1})$ in the direction $\eta(u_2,\dots,u_{n-1})$. Then there is an open set $U_0\subseteq\reals^n$ such that $x:U_0\to U$ is a parameterisation of some neighbourhood $U\subseteq X$ of $x_0$ and $\idiff{x}{u_1}(0,\dots,0) = V$. We shall write $u=(u_1,\dots,u_{n-1})\in U_0$ for short.
	
	We define \[y_t(u) = \pr_1\circ \flow_t\left(x(u),\Pdiff{x}{u_1}(u)\right),\]
	where $\pr_1$ is the projection from $TM$ onto $M$. Then we claim that $ y(u) = y_{\varepsilon-u_1}(u)$ is a smooth parameterisation of a strictly convex submanifold $Y = \{y(u):u\in (-\delta,\delta)\times U_0^*\}$ for sufficiently small constants $\varepsilon,\delta > 0$, and that any smooth geodesic from $y(u)$ in the inward normal direction of $Y$ will intersect $U$ tangentially.
	
	First, we show that the latter claim holds, that is, we show that $\idiff{y}{u_i}$ is normal to $\idiff{y_t}{t}$ for all $i = 1, \dots, n-1$. Note that for $i = 2,\dots,n-1$ we have 
	\[
		\Pdiff{y}{u_i} = \Pdiff{y_{t}}{u_i}\vbar_{t = \varepsilon-u_1} \quad \text{and} \quad \Pdiff{y}{u_1} = \left(\Pdiff{y_t}{u_1} - \Pdiff{y_t}{t}\right)\vbar_{t = \varepsilon-u_1}.
	\]
	Now consider the function
	\[
		f_i = \left\langle\Pdiff{y_t}{u_i}, \Pdiff{y_t}{t} \right\rangle.
	\]
	Then taking derivatives with respect to $t$ we get,
	\begin{align*}
		\dot f_i &= \left \langle \cdiff{t}\Pdiff{y_t}{u_i}, \Pdiff{y_t}{t}\right\rangle\\
		&= \frac{1}{2} \pdiff{u_i}\norm{\Pdiff{y_t}{t}}^2 = 0.
	\end{align*}
	Thus $f_i$ is constant with respect to $t$, so we may compute its value at $t=0$,
	\[
		f_i = \left\langle\Pdiff{x}{u_i}, \Pdiff{x}{u_1} \right\rangle.
	\]
	That is, $f_1 = 1$ and $f_i = 0$ for $i = 2,\dots,n-1$. We also note that for all $i = 2,\dots,n-1$, we have
	\[
		\left\langle \Pdiff{y}{u_i}, \Pdiff{y_t}{t} \right\rangle = f_i = 0.
	\]
	Hence our claim holds for $i=2,\dots,n-1$, leaving only the case where $i=1$.
	But a simple computation shows that the claim holds in this case as well, as follows,
	\[
		\left\langle \Pdiff{y}{u_1}, \Pdiff{y_t}{t}\right\rangle = \left\langle \Pdiff{y_t}{u_1}, \Pdiff{y_t}{t}\right\rangle - \norm{\Pdiff{y_t}{t}}^2 = f_1 - 1 = 0.
	\]
	Hence $\idiff{y}{u_i}$ is normal to $\idiff{y_t}{t}$ for all $i = 1, \dots, n-1$.
	Our next claim is that $y$ is a smooth parameterisation of a submanifold $Y$ in $M$. To prove that the claim holds, it suffices to show that the vectors $\{\idiff{y}{u_i}(0)\}_{i=1}^{n-1}$ are linearly independent when $\varepsilon>0$ is sufficiently small. Suppose this does not hold, then there is some $\overline\varepsilon>0$ such that for all $0<\varepsilon\leq\overline\varepsilon$ the vectors $\{\idiff{y}{u_i}(0)\}_{i=1}^{n-1}$ are linearly dependent. Pick a normal coordinate chart about $y(0)$, and let $\{E_1,\dots,E_n\}$ be the orthonormal frame associated with this chart. Recall that $\partial B_\varepsilon(x_0^*)$ is a strictly convex submanifold of $\partial K$ with a unit normal given by $\idiff{x}{u_1}$. Therefore,
	\begin{equation}\label{eq:x1_normal_to_xi}
		\left\langle \Pdiff{x}{u_1},\Pdiff{x}{u_i}\right\rangle = 0 \textrm{ for all } i = 2,\dots,n-1.
	\end{equation}
	Also recall that $x$ is a unit-speed geodesic in the $u_1$-direction, i.e. $\norm{\idiff{x}{u_1}} = 1$. Hence, 
	\begin{equation}\label{eq:x1j_normal_to_x_1}
		\left\langle \nabla_{\Pdiff{x}{u_j}}\Pdiff{x}{u_1},\Pdiff{x}{u_1}\right\rangle = 0 \textrm{ for all } j = 1,\dots,n-1.
	\end{equation}
	Furthermore, differentiating \cref{eq:x1_normal_to_xi} with respect to $u_1$ we obtain the following identity,
	\begin{equation}\label{eq:x1j_and_xij_cancel}
		\left\langle \nabla_{\Pdiff{x}{u_1}}\Pdiff{x}{u_1},\Pdiff{x}{u_i}\right\rangle + \left\langle \Pdiff{x}{u_1},\nabla_{\Pdiff{x}{u_1}}\Pdiff{x}{u_i}\right\rangle = 0 \textrm{ for all } i = 2,\dots,n-1.
	\end{equation}
	Now combining the two equations, we find the following
	\begin{equation*}
		\left\langle \nabla_{\Pdiff{x}{u_1}}\Pdiff{x}{u_1},\Pdiff{x}{u_i}\right\rangle = 0 \textrm{ for all } i = 1,\dots,n-1.
	\end{equation*}
	The case where $i=1$ follows directly from \cref{eq:x1j_normal_to_x_1}, by setting $j=1$. Thus we have shown that $\nabla_{\idiff{x}{u_1}}\idiff{x}{u_1}$ is a normal to $\partial K$. We should remark here that it is an inward pointing normal. Since $\partial K$ is strictly convex, it follows that $\norm{\nabla_{\idiff{x}{u_1}}\idiff{x}{u_1}}>0$. Furthermore, since the vectors $\{\idiff{x}{u_i}(0)\}_{i=2}^{n-1}$ are linearly independent, and $\nabla_{\idiff{x}{u_1}}\idiff{x}{u_1}$ is orthogonal to $\idiff{x}{u_i}$ for all $i = 2,\dots, n-1$, it follows that \[\mathcal{L} = \left\{\nabla_{\Pdiff{x}{u_1}}\Pdiff{x}{u_1}\right\}\cup\left\{\Pdiff{x}{u_i}\right\}_{i=2}^{n-1},\] is also a linearly independent set. We then write the parameterisation $x(u)$ locally as $(x_1(u),\dots,x_n(u))$, and $\nabla_{\idiff{x}{u_1}}\idiff{x}{u_1} = (X_1,\dots,X_{n})$. Consider the matrix $B$ whose rows are the vectors in $\mathcal{L}$,
	\[B = \left[
	\begin{matrix}
	  X_1(0) & \dots & X_n(0) \\
	  \idiff{x_1}{u_2}(0) & \dots & \idiff{x_n}{u_2}(0) \\
	  \vdots & \ddots & \vdots \\
	  \idiff{x_1}{u_{n-1}}(0) & \dots & \idiff{x_n}{u_{n-1}}(0)
	\end{matrix}
	\right].\]
	$B$ has rank $n-1$, since its rows are linearly independent. Thus there is a column we can remove from $B$ while maintaining its rank. Without loss of generality, assume we may remove the last column. Then the submatrix obtained by removing the last column of B,
	\[
		B' = \left[
	\begin{matrix}
	  X_1(0) & \dots & X_{n-1}(0) \\
	  \idiff{x_1}{u_2}(0) & \dots & \idiff{x_n}{u_2}(0) \\
	  \vdots & \ddots & \vdots \\
	  \idiff{x_1}{u_{n-1}}(0) & \dots & \idiff{x_{n-1}}{u_{n-1}}(0)
	\end{matrix}
	\right],
	\]
	has nonzero determinant. We will use the fact that $\det B' \neq 0$ to arrive at a contradiction to the assumption that the vectors $\{\idiff{y}{u_i}(0)\}_{i=1}^{n-1}$ are linearly dependent for all ${0<\varepsilon\leq\overline\varepsilon}$. We write $y(u)$ locally as $(y_1(u),\dots,y_n(u))$ in our normal coordinate chart. Consider the square matrix
	\[A = \left[
	\begin{matrix}
	  \idiff{y_1}{u_1}(0) & \dots & \idiff{y_{n-1}}{u_1}(0) \\
	  \vdots & \ddots & \vdots \\
	  \idiff{y_1}{u_{n-1}}(0) & \dots & \idiff{y_{n-1}}{u_{n-1}}(0)
	\end{matrix}
	\right].\]
	Since the vectors $\{\idiff{y}{u_i}\}_{i=1}^{n-1}$ are linearly dependent, so are the rows of $A$. Thus $\det A = 0$. Before we proceed to showing a contradiction, we must examine $\idiff{y}{u_i}$ more closely. Working within our normal coordinate neighbourhood, consider the Taylor expansion of $y$,
	\begin{equation}\label{eq:y_taylor_convex_front_construction}
		y(u) = x(u) + (\varepsilon-u_1)\Pdiff{x}{u_1} - \frac{1}{2}(\varepsilon-u_1)^2\sum_{i,j,k=1}^{n}\Gamma_{ij}^k(u)\Pdiff{x_i}{u_1}\Pdiff{x_j}{u_1}E_k + O((\varepsilon-u_1)^3).
	\end{equation}
	We will make use of this expansion in the latter stages of the proof as well. For now we highlight two consequences,
	\begin{align*}
		\Pdiff{y}{u_1} &= \Pdiff{x}{u_1} - \Pdiff{x}{u_1} + (\varepsilon - u_1) \Pdifftwo{x}{u_1} + (\varepsilon-u_1)\sum_{i,j,k=1}^{n}\Gamma_{ij}^k(u)\Pdiff{x_i}{u_1}\Pdiff{x_j}{u_1}E_k + O((\varepsilon-u_1)^2)\\
		&= (\varepsilon - u_1) \nabla_{\Pdiff{x}{u_1}}\Pdiff{x}{u_1} + O((\varepsilon-u_1)^2).
	\end{align*}
	Hence,
	\begin{equation}\label{eq:der_y1}
		\Pdiff{y}{u_1}(0) = \varepsilon\nabla_{\Pdiff{x}{u_1}}\Pdiff{x}{u_1}(0) + O(\varepsilon^2),
	\end{equation}
	\begin{equation}\label{eq:der_yi}
		\Pdiff{y}{u_i}(0) = \Pdiff{x}{u_i} + O(\varepsilon) \textrm{ for all } i = 2,\dots,n-1.
	\end{equation}
	Now consider the matrix $A'$ obtained by dividing the first row of $A$ by $\varepsilon$,
	\[
		A' = \left[
		\begin{matrix}
		  X_1(0) + O(\varepsilon) & \dots & X_{n-1}(0) + O(\varepsilon)  \\
		  \idiff{x_1}{u_2}(0) + O(\varepsilon)  & \dots & \idiff{x_n}{u_2}(0) + O(\varepsilon)  \\
		  \vdots & \ddots & \vdots \\
		  \idiff{x_1}{u_{n-1}}(0) + O(\varepsilon)  & \dots & \idiff{x_{n-1}}{u_{n-1}}(0) + O(\varepsilon) 
		\end{matrix}
		\right].
	\]
	Notice that $\lim_{\varepsilon\to 0} A' = B'$. However,
	\[
		\det B' = \lim_{\varepsilon\to 0} \det A' = (\lim_{\varepsilon\to 0} \varepsilon^{-1})\det A = 0.
	\]
	This is a contradiction since, as we have shown, $\det B' \neq 0$. Therefore $y$ is indeed a smooth parameterisation of a smooth submanifold $Y$ in $M$, for some sufficiently small $\varepsilon>0$.
	
	To conclude the proof we must now show the existence of a positive lower bound on the curvature of $Y$. We will continue to work in the normal coordinate chart about $y(0)$ as before. We denote the unit normal field to $y(u)$ as $N(u)$ and the shape operator of $Y$ by $s_Y$, then $s_Y W = \nabla_{W} N$ for all vectors $W$ tangent to $Y$. Since the vectors $\idiff{y}{u_i}(0)$ are linearly independent, it suffices to calculate the terms
	\[
		s_Y^{ij}=\left\langle s_Y\frac{\idiff{y}{u_i}(0)}{\norm{\idiff{y}{u_i}(0)}}, \frac{\idiff{y}{u_j}(0)}{\norm{\idiff{y}{u_j}(0)}} \right\rangle,
	\]
	and show that there is a positive lower bound for all $i,j = 1,\dots,n-1$.
	 As we have shown prior, we know that $N(u) = \idiff{y_t}{t}\vert_{t=\varepsilon-u_1}$. Thus, by \cref{eq:y_taylor_convex_front_construction},
	\begin{equation*}
		N(u) = \Pdiff{x}{u_1} - (\varepsilon-u_1)\sum_{i,j,k=1}^{n}\Gamma_{ij}^k(u)\Pdiff{x_i}{u_1}\Pdiff{x_j}{u_1}E_k + O((\varepsilon-u_1)^2).
	\end{equation*}
	Now at $u=0$ for $i,j = 2,\dots,n-1$, by \cref{eq:der_yi} we have
%	\begin{equation*}
%		\Pdiff{y}{u_i} = \Pdiff{x}{u_i} + O(\varepsilon - u_1)
%	\end{equation*}
	\begin{equation*}
		s_Y\Pdiff{y}{u_i}(0) = \nabla_{\Pdiff{y}{u_i}}N(0) = \nabla_{\Pdiff{x}{u_i}}\Pdiff{x}{u_1}(0) + O(\varepsilon),
	\end{equation*}
	\begin{equation*}
		\left\langle \nabla_{\Pdiff{y}{u_i}}N(0), \Pdiff{y}{u_j}(0) \right\rangle = \left\langle \nabla_{\Pdiff{x}{u_i}}\Pdiff{x}{u_1}(0), \Pdiff{x}{u_j}(0) \right\rangle + O(\varepsilon)>K_B+O(\varepsilon),
	\end{equation*}
	\begin{equation}\label{eq:norm_der_yi}
		\norm{\Pdiff{y}{u_i}(0)} = \norm{\Pdiff{x}{u_i} + O(\varepsilon)} = 1 + O(\varepsilon^{1/2}).
	\end{equation}
	Recall that $\partial B_\varepsilon(x_0^*)$ was a strictly convex submanifold with unit normal $\idiff{x}{u_1}$. Hence at $u=0$, provided we choose $\varepsilon>0$ to be sufficiently small, for any $K_Y>0$ we have
	\[
		s_Y^{ij}=\left\langle s_Y\frac{\idiff{y}{u_i}(0)}{\norm{\idiff{y}{u_i}(0)}}, \frac{\idiff{y}{u_j}(0)}{\norm{\idiff{y}{u_j}(0)}} \right\rangle  > K_Y \textrm{ for all } i,j=2,\dots,n-1.
	\]
	We will now proceed to show that there exists a lower bound for $s_Y^{i1}$, for all $i = 2,\dots, n-1$. Using \cref{eq:der_y1} and \cref{eq:der_yi}, one can find the following,
	\[
		\left\langle s_Y\Pdiff{y}{u_i}(0),\Pdiff{y}{u_1}(0)\right\rangle = - \left\langle \nabla_{\Pdiff{y}{u_i}}\Pdiff{y}{u_1}(0), N(0) \right\rangle = -\varepsilon\left\langle \nabla_{\Pdiff{x}{u_i}}\nabla_{\Pdiff{x}{u_1}}\Pdiff{x}{u_1}(0), \Pdiff{x}{u_1}\right\rangle + O(\varepsilon^2).
	\]
	Recall that $\nabla_{\idiff{x}{u_1}}\idiff{x}{u_1}$ is an inward normal to $\partial K$. Hence
	\[
		\left\langle s_Y\Pdiff{y}{u_i}(0),\Pdiff{y}{u_1}(0)\right\rangle > \varepsilon\mincurv \norm{\nabla_{\Pdiff{x}{u_1}}\Pdiff{x}{u_1}(0)} + O(\varepsilon^2).
	\]
	Where $\mincurv>0$ is the lower bound on the curvature of $\partial K$. Once again using \cref{eq:der_y1}, along with \cref{eq:norm_der_yi}
	\begin{equation}\label{eq:norm_der_y1}
		\norm{\Pdiff{y}{u_1}(0)}^2 = \varepsilon^2\norm{\nabla_{\Pdiff{x}{u_1}}\Pdiff{x}{u_1}(0)}^2 + O(\varepsilon^3).
	\end{equation}
	\[
		\left\langle s_Y\frac{\idiff{y}{u_i}(0)}{\norm{\idiff{y}{u_i}(0)}}, \frac{\idiff{y}{u_1}(0)}{\norm{\idiff{y}{u_1}(0)}} \right\rangle > \frac{\mincurv\norm{\nabla_{\idiff{x}{u_1}}\idiff{x}{u_1}(0)} + O(\varepsilon)}{\sqrt{\norm{\nabla_{\idiff{x}{u_1}}\idiff{x}{u_1}(0)}^2 + O(\varepsilon)}} \to \mincurv \textrm{ as } \varepsilon\to 0.
	\]
	Therefore, provided $\varepsilon$ is sufficiently small, $s_Y^{i1}$ ($i=2,\dots,n-1$) is bounded below by some $K_Y>0$ which depends only on the minimum principal curvature $\mincurv$ of $\partial K$ and our choice of $\varepsilon$.
	For the last curvature term, $s_Y^{11}$, we compute the derivative of $N$,
	\begin{align*}
		\Pdiff{N}{u_1} &= \Pdifftwo{x}{u_1} + \sum_{i,j,k=1}^{n}\Gamma_{ij}^k(u)\Pdiff{x_i}{u_1}\Pdiff{x_j}{u_1}E_k + O(\varepsilon-u_1)\\
			&= \nabla_{\Pdiff{x}{u_1}}\Pdiff{x}{u_1} + O(\varepsilon-u_1).
	\end{align*}
	Then at $u = 0$ we have,
	\begin{equation*}
		\left\langle \Pdiff{N}{u_1}(0), \Pdiff{y}{u_1}(0) \right\rangle = \varepsilon \norm{\nabla_{\Pdiff{x}{u_1}}\Pdiff{x}{u_1}(0)}^2 + O(\varepsilon^2).
	\end{equation*}
	Now recall that we picked normal coordinates about $y(0)$, so that we have
	\[
		\left\langle s_Y\Pdiff{y}{u_1}(0), \Pdiff{y}{u_1}(0) \right\rangle = \left\langle \nabla_{\Pdiff{y}{u_1}}N(0), \Pdiff{y}{u_1}(0) \right\rangle = \left\langle \Pdiff{N}{u_1}(0), \Pdiff{y}{u_1}(0) \right\rangle.
	\]
	Using \cref{eq:norm_der_y1} we may conclude the following
	\[
		\left\langle s_Y\frac{\idiff{y}{u_1}(0)}{\norm{\idiff{y}{u_1}(0)}}, \frac{\idiff{y}{u_1}(0)}{\norm{\idiff{y}{u_1}(0)}} \right\rangle = \frac{1}{\varepsilon}\frac{\norm{\nabla_{\idiff{x}{u_1}}\idiff{x}{u_1}(0)}^2 + O(\varepsilon)}{\norm{\nabla_{\idiff{x}{u_1}}\idiff{x}{u_1}(0)}^2 + O(\varepsilon)}.
	\]
	Therefore if our previous choice of $\varepsilon$ was not small enough, we may shrink $\varepsilon$ so that, $s_Y^{11}>K_Y$. We note that $K_Y$ depends only on $\mincurv$ and our choice of $\varepsilon$. Hence we have shown that $Y$ is strictly convex, i.e. its curvature is positive and bounded below by $K_Y$ depending only on $\mincurv$ and our choice of $\varepsilon$. Therefore, since $\partial K$ is compact, there exists a global lower bound $\Theta_0 > 0$, for the curvature of $Y$ for all $x_0$ and $V$.
\end{proof}

\section{Proof of \Cref{Theorem 1}}

Combining all the results in \Cref{sect:propagation} shows that any strictly convex front constructed via \Cref{proposition:convex_front_from_obstacle} will remain strictly convex when propagated forward via the billiard flow. This follows since the front may travel for distance at most $\maxdist$ between each reflection, and by \Cref{lemma:always_convex} it will remain strictly convex for a distance of at least $\maxdist\xi$. Then every $\minref$ reflections, by \Cref{lemma:minimal_reflection_angle}, the front will hit $\partial K$ at an angle of at most $\varphi_0$, at least once. By \Cref{lemma:reflection_curvature} the principal curvatures of the front will therefore be bounded below by $\globalcurv = \min\{2\mincurv\cos\varphi_0,\globalcurv_0\}$ once more. In this section we leverage these facts to prove a few additional results that will be useful in the proof of \Cref{Theorem 1}, which is given at the end of the section.

\begin{proposition}\label{proposition:convex_front_collision}
	Suppose $X$ and $Y$ are two strictly convex fronts such that for some $t_0>0$ and $x\in X$ we have $x_{t_0}=\gamma_{x,N(x)}(t_0)\in Y$. Moreover, suppose that $N_{t_0}(x) = \dot\gamma_{x,N(x)}(t_0)$ is an inward normal to $Y$, and that for any principal curvatures $k_X$ and $k_Y$ of $X$ and $Y$ respectively with respect to $N_{t_0}(x)$ at $x_{t_0}$, we have $k_Y<-\frontmincurv<0<\frontmincurv<k_X$ . Let $\mathcal{G}$ be the set of points $x^*\in X$ such that $x^*_{t(x^*)}\in Y$ and $N_{t(x^*)}(x^*)$ is normal to $Y$. Then $\mathcal{G}$ is a submanifold of $X$ of dimension 0. That is, $\mathcal{G}$ is at most countable.
\end{proposition}
\begin{proof}
	Shrink $X$ and $Y$ such that for all $x\in X$ there is a $t(x)>0$ such that $x_{t(x)} \in Y$.
	We begin by taking Fermi coordinates about $X$ (cf. chapter 5 in \cite{LEERIEMANN}). Then for any $x\in X$ we have $x = (x_1,\dots, x_{n-1}, 0)$, and $y(x) = x_{t(x)} = (x_1,\dots,x_{n-1},t(x))\in Y$. Note that the outward unit normal to $X$ in these coordinates is $N = (0,\dots,0,1)$. Let $g_i(x) = \langle\idiff{y}{x_i}, N\rangle$ for each $i = 1,\dots n-1$. Then we have
	\[
		\Pdiff{y}{x_i} = \Pdiff{x}{x_i} + \Pdiff{t}{x_i}N.
	\]
	Hence, taking inner products with $N$ on both sides, we get
	\[
		\Pdiff{t}{x_i} = \langle \Pdiff{y}{x_i}, N \rangle = g_i.
	\]
	Now taking the covariant derivative we get
	\[
		\nabla_{\Pdiff{y}{x_i}}\Pdiff{y}{x_j} = \nabla_{\Pdiff{y}{x_i}}\Pdiff{x}{x_j} + \Pdifftwov{t}{x_i}{x_j}N + \nabla_{\Pdiff{y}{x_i}} N.
	\]
	Since $\nabla_N N = 0$ we know that
	\[
		\langle \nabla_{N}\Pdiff{x}{x_j}, N\rangle = N\langle \Pdiff{x}{x_j}, N \rangle = 0.
	\]
	Therefore we note that 
	\[
		\langle \nabla_{\Pdiff{y}{x_i}}\Pdiff{x}{x_j}, N\rangle = \langle \nabla_{\Pdiff{x}{x_i}}\Pdiff{x}{x_j} + \Pdiff{t}{x_i}\nabla_N\Pdiff{x}{x_j}, N\rangle = \langle \nabla_{\Pdiff{x}{x_i}}\Pdiff{x}{x_j}, N\rangle.
	\]
	Allowing us to conclude,
	\[
		\Pdifftwov{t}{x_i}{x_j} = \langle \nabla_{\Pdiff{y}{x_i}}\Pdiff{y}{x_j}, N \rangle - \langle \nabla_{\Pdiff{x}{x_i}}\Pdiff{x}{x_j}, N\rangle.
	\]
	Now set $g = (g_1,\dots,g_{n-1}):X\to \reals^{n-1}$, and let $\mathcal{G} = g^{-1}(0)$. Given $x_0\in \mathcal{G}$, we know that $N$ is the unit inward normal to $Y$ at $y(x_0)$. Hence,
	\[
		\Pdiff{g_i}{x_j} = \Pdifftwov{t}{x_i}{x_j}(x_0) > 2\frontmincurv > 0. 
	\]
	So $dg_{x_0}$ is surjective, and $\mathcal{G} = g^{-1}(0)$ is a 0-dimensional submanifold of $X$.
\end{proof}

\begin{proposition}\label{proposition:distance_bound_on_front}
	Let $q:[0,\ell]\to X$ be an arc-length parameterised curve on a front $X$. Denote by $q_t(u)\in X_t$ the curve $q$ after propagation along the geodesic flow for time $t>0$ in the normal direction. Then
	\begin{equation}\label{eq:dist_on_front_bound}
		d_{X_t}(q_t(0),q_t(\ell)) \geq d_X(q(0),q(\ell))e^{tk_{\min}}.
	\end{equation}
	Where $k_{\min}$ is the minimum principal curvature of $X_t$, and $d_X$ and $d_{X_t}$ are the distance functions on $X$ and $X_t$ respectively.
\end{proposition}
\begin{proof}
	We begin by noting that $J_u(t) = \pdiff{u}q_t(u)$ is a Jacobi field for each $u\in [0,\ell]$, since $q_t(u)$ is a variation through geodesics. We also note that
	\begin{align*}
		\cdiff{t}J_u(t) &= \cdiff{t}\pdiff{u}q_t(u)\\
		&= \cdiff{u}\pdiff{t}q_t(u)\\
		&= s_t J_u(t).
	\end{align*}
	Here $s_t$ is the shape operator of $X_t$. To bound $d(q_t(0),q_t(\ell))$ we are interested in the norm of $J_s(t)$, so consider the function $f(u,t) = \norm{J_u(t)}$. We take the derivative of $f$ in $t$ as follows:
	\begin{align*}
		\dot f(t) &= \frac{1}{\norm{J_u(t)}}\left\langle \cdiff{t}J_u(t), J_u(t) \right\rangle\\
		&= \frac{1}{\norm{J_u(t)}}\left\langle s_t J_u(t), J_u(t) \right\rangle\\
		&= \norm{J_u(t)}\left\langle s_t \frac{J_u(t)}{\norm{J_u(t)}}, \frac{J_u(t)}{\norm{J_u(t)}} \right\rangle\\
		&= f(t)\left\langle s_t\frac{J_u(t)}{\norm{J_u(t)}}, \frac{J_u(t)}{\norm{J_u(t)}} \right\rangle.
	\end{align*}
	It now follows that the solution to $f(t)$ is,
	\[
		f(t) = \exp\left(\int^t_0 \left\langle s_t\frac{J_u(t)}{\norm{J_u(t)}}, \frac{J_u(t)}{\norm{J_u(t)}} \right\rangle dt \right)\geq e^{tk_{\min}}.
	\]
	Suppose now that for a fixed $t>0$, the curve $q_t(u)$ is the minimal geodesic along $X$ between $q_t(0)$ and $q_t(\ell)$.
	\begin{align*}
		d_{X_t}(q_t(0),q_t(\ell)) &= \int_0^\ell f(t)\ du \geq \int_0^\ell e^{tk_{\min}}\ du = \ell e^{tk_{\min}}.
	\end{align*}
	This is \cref{eq:dist_on_front_bound}, as required.
\end{proof}
Owing to \cref{lemma:reflection_curvature} and \cref{proposition:distance_bound_on_front} above, we get the following corollary,
\begin{corollary}\label{lemma:dist_on_front_reflection}
	Let $q_t$ be as in \cref{proposition:distance_bound_on_front}, but under propagation along the billiard flow. That is, we allow for reflections. Suppose $q_t$ reflects at $0<t_1<t_2<\dots <t_n$ then
	\begin{equation}
		d_{X_t}(q_t(0),q_t(\ell)) \geq d_X(q(0),q(\ell))e^{t_n k_{\min}}.
	\end{equation}
\end{corollary}

\newpage

We define the following submanifold of $TS$:
\[
	\T S_K = \{(x,v)\in TS:x\in \partial K\}\cup\TS
\]
\begin{proposition}\label{proposition:tangent_twice}
	There exists a countable family $\{\Xi_i\}$ of codimension 2 smooth submanifolds of $\T S_K$ such that for any $\sigma\in\T S_K\backslash (\cup_i\Xi_i)$ the billiard ray generated by $\sigma$ is tangent to $\partial K$ at most once.
\end{proposition}
\begin{proof}
	First, suppose that $\sigma_0\in \uT{\partial K}$ is such that there is some $t_0>0$ and $\sigma_1\in\uT{\partial K}$ such that $\genflow_{t_0}(\sigma_0) = \sigma_1$. Let $0<i_1,i_2\leq d$ be integers such that $\pr_1(\sigma_0)\in\partial K_{i_1}$ and $\pr_1(\sigma_0)\in\partial K_{i_2}$.
	Let $\widetilde V$ be an open neighbourhood of $\sigma_0$ in $\uT{S}$, such that $\pr_1(\widetilde V)\cap\partial K_i=\emptyset$ for all $i\neq i_1$. We set $V = \pr_1(\widetilde V)$ and construct a strictly convex front $X$, via \Cref{proposition:convex_front_from_obstacle}, such that any smooth geodesic from $X$ in the inward normal direction to $X$ will intersect $\partial K_{i_1}$ tangentially. In particular, let $\gamma$ be the smooth geodesic from $x_0$ in the direction $\sigma_0$ and denote the intersection of $\gamma$ and $X$ by $x_0'$. Also let $\varepsilon_0>0$ be the time such that $\pr_1\circ\flow_{\varepsilon_0}(\sigma_0)=x_0'$. Note that we may shrink $\widetilde V$ and $\varepsilon_0$ as needed, shrinking $X$ and bringing it closer to $x_0$ in the process, while ensuring that $X\cap \partial K = \emptyset$. Denote the outward unit normal to $X$ at $x$ by $N_X(x)$, and let
	\[
		\widetilde X = \{(x,N_X(x)):x\in X\},
	\]
	\[
		\widehat X = \{(x,v):x\in X,v\in TS\}.
	\]
	Note that $\widetilde X$ is an $n-1$ dimensional submanifold of $TS$ while $\widehat X$ is a submanifold of codimension 1 in $TS$. Possibly shrinking $\widetilde V$, there is a smooth function $g:\widetilde V\to\reals$ such that $\flow_{g(\sigma)}(\sigma)\in\widehat X$ for each $\sigma \in \widetilde V$. Now consider the map $f:\widetilde V\to\widehat X$ given by $\sigma\mapsto\flow_1(g(\sigma)\sigma)$, recalling that $\norm{\sigma}=1$ for all $\sigma\in\widetilde V$ it follows that $\norm{f(\sigma)}=g(\sigma)$. Note that $f$ is a diffeomorphism between $\widetilde V$ and $\widehat X$. We shall now define similar neighbourhoods for $\sigma_1$. Let $\widetilde W$ be an open neighbourhood of $\sigma_1$ in $\uT S$, and set $W=\pr_1(\widetilde W)$. Shrink $\widetilde W$ so that $W\cap\partial K_i =\emptyset$ for all $i\neq i_2$. 
	
	Let $\widetilde K = K \backslash (K_{i_1}\cup K_{i_2})$.
	 Define the curves $\chi:[0,t_0+\mindist/2]\times\widetilde V\to \uT S$ as the billiard rays in $S_{\widetilde K}$ generated by $\sigma\in\widetilde V$. That is, $\chi_t(\sigma)$ is the billiard ray generated by $\sigma$, which ignores reflections on $\partial K_{i_1}$ and $\partial K_{i_2}$. Pick $t_0^*>0$ such that $|t_0-t_0^*|<\mindist/2$, and let $\widetilde Y = \chi_{t_0^*}(\widetilde X)$. Let $Y = \pr_1(\widetilde Y)$, and set
	 \[
	 	\widehat Y = \{(y,v):y\in Y,v\in TS\}.
	 \]
	Once again, shrinking $\widetilde V$ or $\widetilde W$ if necessary, there is a smooth function $\mathring{g}: \widetilde W\to\reals$ such that $\flow_{\mathring{g}(\sigma)}(\sigma)\in\widehat Y$ for all $\widetilde W$. Then we may define the diffeomorphism $\mathring{f}:\widetilde W\to\widehat Y$ given by $\sigma\mapsto\flow_1(\mathring{g}(\sigma)\sigma)$. Thus we have a diffeomorphism $\Phi = \mathring{f}^{-1}\circ\chi_{t_0^*}\circ f:\widetilde V\to\widetilde W$, and in particular $\Phi(\sigma_0)=\sigma_1$.
	
	Let $\varphi_V:\widetilde V\to\reals^2$ and $\varphi_W:\widetilde W\to\reals^2$ be local defining functions for $\uT\partial K$ in $\widetilde V$ and $\widetilde W$ respectively. Denote $V_0 = \varphi_V^{-1}(0)$ and $W_0 = \varphi_W^{-1}(0)$. We shall show that $V_0\cap\Phi^{-1}(W_0)$ has codimension at least 1 in $V_0$. Let $\psi_V:V\to\reals$ and $\psi_W:W\to\reals$ be local defining functions for $\partial K$ in $V$ and $W$ respectively. It follows that $\grad\psi_V/\norm{\grad\psi_V}$ is the outward unit normal to $\partial K$ on $V\cap\partial K$ and similarly for $W\cap\partial K$. Consider the sets $(d\psi_V)^{-1}(0),(d\psi_W)^{-1}(0)\subseteq TS$. Note that 0 is a regular value of $d\psi_V$, since
	\[
		\nabla (d\psi_V) (X,Y) = -\norm{\grad \psi_V}h_{\partial K}(X,Y) \geq 0\textrm{ for all }X,Y\in \mathfrak{X}(\partial K),
	\]
	where $h_{\partial K}$ is the scalar second fundamental form of $\partial K$. Thus $(d\psi_V)^{-1}(0)$ is a codimension 1 submanifold of $TS$, and similarly for $(d\psi_W)^{-1}(0)$. Now set $V' = (d\psi_V)^{-1}(0)\cap \widetilde V$ and $W' = (d\psi_W)^{-1}(0)\cap \widetilde W$. Note that both $V'$ and $W'$ are codimension 1 submanifolds of $\widetilde V$ and $\widetilde W$.
	Consider the codimension 1 submanifold $Y'$ of $\widehat Y$ given by $\mathring{f}(W')$.
	We claim that $Y'$ and $\widetilde Y$ are transversal near $(y_0,\mu_0) = \chi_{t^*_0}\circ f (\sigma_0)$. Denote the outward unit normal of $Y$ at $y$ by $\mu(y)$. Let $\widehat t_0 = t_0-t^*_0$, and $N^*_0$ be the outward unit normal of $\partial K_{i_2}$ at $x_1$. Denote by $N^*:[t_0^*,t_0]\to\uT S$ the vector field given by parallel translation of $N^*_0$ along the smooth geodesic in $S$ from $y_0$ to $x_1$ such that $N^*(t_0) = N^*_0$. For some small $\delta>0$, let $\lambda:(-\delta,\delta)\to Y$ be the smooth geodesic in $Y$ such that $\lambda(0) = y_0$ and $\lambda'(0) = N^*(t_0^*)$. Note that $N^*(t_0^*)$ is indeed tangent to $Y$ by construction. Let $\omega(s) = (\lambda(s),\mu(\lambda(s)))$, and $p(s) = \pr_1\circ\flow_{\widehat t_0}(\omega(s))$. Then $p((-\delta,\delta))$ is diffeomorphic to $\lambda((-\delta,\delta))$ via the shift along the geodesic flow. Indeed, suppose there exist distinct points $\omega_1,\omega_2\in\widehat Y$ such that $\pr_1\circ\flow_{\widehat t_0}(\omega_1)=\pr_1\circ\flow_{\widehat t_0}(\omega_2)$. Owing to \Cref{lemma:always_convex}, \Cref{lemma:reflection_curvature} and \Cref{lemma:minimal_reflection_angle}, $Y$ is strictly convex. Thus by \Cref{proposition:distance_bound_on_front}, we must have 
	\[d(\pr_1\circ\flow_{\widehat t_0}(\omega_1),\pr_1\circ\flow_{\widehat t_0}(\omega_2))>0,\]
	contradicting our assumption. Now for $\delta>0$ sufficiently small, since $p'(0)=N^*_0$, we have $p((-\delta,0))\subseteq K_{i_2}$ and $p((0,\delta))\cap K_{i_2}=\emptyset$. Thus for $s\in(-\delta,0)$, the smooth geodesic
	\[
		r(s) = \{\pr_1\circ\flow_t(\omega(s)):0\leq t\leq \widehat t_0\},
	\]
	must intersect $\partial K_{i_2}$. Furthermore, since $\partial K_{i_2}$ is strictly convex, provided $r(s)$ remains sufficiently close to $x_1$, the smooth geodesic $r(s)$ cannot be tangent to $\partial K_{i_2}$. For $s\in(0,\delta)$, since $r(0)$ is tangent to $\partial K_{i_2}$, which is strictly convex, and the Jacobi field of $r(s)$ at $r(0)$ is precisely $N^*$ by construction, it follows that for sufficiently small $\delta$, the smooth geodesic $r(s)$ cannot intersect $\partial K_{i_2}$ at all. Thus the curve $\omega(s)$ was constructed such that $\omega(s)\in\widetilde Y$, $\omega(0)=(y_0,\mu_0)\in Y'$ and $\omega(s)$ is transversal to $Y'$ in $\widehat Y$. Now since $Y'$ has codimension 1 in $\widehat Y$, this is sufficient to claim that near $(y_0,\mu_0)$ the submanifolds $\widetilde Y$ and $Y'$ are transversal as desired. It follows that $Y'\cap\widetilde Y$ has codimension 1 in $\widetilde Y$.
	
	Note that $\widetilde Z = f^{-1}\circ\chi_{t^*_0}^{-1}(\widetilde Y)$ is an $n-1$ dimensional submanifold of $V_0$ by construction. Now consider the $n-2$ dimensional submanifold $Z' = f^{-1}\circ\chi_{t^*_0}^{-1}(Y'\cap\widetilde Y)$. Since $Z'$ has codimension 1 in $\widetilde Z$, we must have $\widetilde Z\backslash Z'\neq\emptyset$. Take any $\sigma\in \widetilde Z\backslash Z'$, then $\Phi(\sigma)\in\widetilde W\backslash W'$ by construction. But $W_0\subseteq W'$, hence it follows that $\sigma\in V_0\backslash \Phi^{-1}(W_0)$. Now owing the the fact that $\widetilde V$ is connected, and since $V_0\cap\Phi^{-1}(W_0)$ is closed, we can conclude that $V_0\cap\Phi^{-1}(W_0)$ must have codimension at least 1 in $V_0$.
	
	Therefore we have shown that there is a submanifold, $V_0\cap\Phi^{-1}(W_0)$, of codimension 2 or greater in $\T{S_K}$ containing $\sigma_0$. Since $\sigma_0$ was arbitrary, this is true for all $\sigma\in\uT{\partial K}$ which generate a ray that is tangent to $\partial K$ more than once. Let $\gamma_\sigma(t) = \genflow_t(\sigma)$ for all $t\in\reals$ and $\sigma\in\T S_K$. Let $\Xi$ denote the set of points $\sigma\in\T S_K$ such that $\gamma_\sigma$ is tangent to $\partial K$ more than once. For any pair of integers $\alpha<\beta$ let $\Xi^\alpha_\beta$ be the subset of $\Xi$ containing all $\sigma$ such that $\gamma_\sigma$ has a tangency after $\alpha$ transversal reflections and a tangency after transversal $\beta$ reflections. Here the number of reflections in the negative time direction is represented as negative integers. That is, if $t_\alpha<0$ is the time such that $\gamma_\sigma(t_\alpha)\in\uT{\partial K}$ then $\alpha<0$, and similarly for $\beta$. Note that the numbers of reflections $\alpha$ and $\beta$ are both counted starting from $\gamma_\sigma(0)$. We will show that for every $\sigma\in\Xi^\alpha_\beta$ there is a neighbourhood $U\subseteq\T S_K$ such that $U\cap\Xi^\alpha_\beta$ can be covered by a countable union of codimension 2 submanifolds of $U$. Note that by our argument above this immediately holds for the case where either $\alpha=0$ or $\beta = 0$. Suppose that both $\alpha\neq 0$ and $\beta\neq 0$. Then given any $\sigma_{-1}\in\Xi_\beta^\alpha$, let $t_\alpha$ and $t_\beta$ be the times of the first two tangencies of $\gamma_\sigma$. Set $\sigma_0 = -\gamma_{\sigma_{-1}}(t_\alpha)$.
	 By our previous argument there is a codimension 2 submanifold $\Psi(\sigma_0)\subseteq\T S_K$ containing $\sigma_0$. Since there are exactly $\alpha$ transversal reflections along $\gamma_{\sigma_{-1}}$ between $\sigma_{-1}$ and $\sigma_0$, there are open neighbourhoods $U(\sigma_0),U(\sigma_{-1})\subseteq\T S_K$ of $\sigma_0$ and $\sigma_{-1}$ respectively, and a diffeomorphism $\Upsilon:U(\sigma_0)\to U(\sigma_{-1})$ via the billiard flow in $S_K$. Recall that $\Psi(\sigma_0) = V_0\cap\Phi^{-1}(W_0)$ via our previous argument and by construction $\Upsilon^{-1}(\Xi_\beta^\alpha\cap U(\sigma_{-1}))\subseteq \Psi(\sigma_0)$. Thus $\Xi_\beta^\alpha\cap U(\sigma_{-1})\subseteq \Upsilon(\Psi(\sigma_0))$, which completes our argument.
\end{proof}

\begin{proposition}\label{proposition:trapped_rays_dimension_zero}
Given an open set $\widetilde V\subseteq T_1\partial K$, let $V$ be the set of points such that for all $(x,\omega)\in V$ the ray $\gamma(x,\omega)$ from $x$ in the direction $\omega$ is trapped, but never tangent. Then $V$ has topological dimension 0.
\end{proposition}
\begin{proof}
	Given an open set $\widetilde V\subseteq T_1\partial K$, let $V'$ be the set of points $(x,\omega)\in\widetilde V$ such that the ray $\gamma(x,\omega)$ from $x$ in the direction $\omega$ is tangent to $K$ at some point $y \neq x$. Possibly shrinking $\widetilde V$, construct a convex front $X$ from $\widetilde V$ as in \Cref{proposition:convex_front_from_obstacle}. Let $\widetilde U$ be the convex front along with its outward unit normal field $N(x)$. So $\widetilde U = \{(x,N(x)):x\in X\}$. Let $V$ be the set of points $(x,\omega)\in \widetilde V\backslash V'$ for which the ray $\gamma(x,\omega)$ is trapped. Then there are corresponding sets $U', U\subseteq \widetilde U$ to $V', V$.
	Let $\widetilde F = \{1,\dots, d\}$ and define the metric space
	\[
		F = \prod_{i=1}^\infty \widetilde F,\quad \eta (a,b) = \sum_{i=1}^\infty\frac{1}{2^i}(1-\delta_{a_ib_i}),
	\]
	with the metric $\eta$ defined for all $a,b \in F$, where $a = (a_1,a_2,\dots)$, $b = (b_1, b_2, \dots)$ and $\delta_{ij}$ is the Kronecker delta. i.e. $\delta_{a_ib_i} = 1$ if $a_i = b_i$ and 0 otherwise. It is well known that $F$ is a 0-dimensional metric space (cf. \cite{MR0482697} pg. 22). We claim that $U$ is homeomorphic to a subset of $F$. We shall prove this claim by constructing a homeomorphism $f: U \to F$ as follows:
	Begin by ordering the components of $K$ in a fixed order, say $K_1,\dots,K_d$. Now given $(x,\omega)\in U$, since $\gamma(x,\omega)$ is trapped, there is a sequence $a_1, a_2, \dots$, with $a_i\in \widetilde F$ such that the $i$'th reflection point of $\gamma(x,\omega)$ is on $K_{a_i}$. We define $f(x,\omega) = (a_1,a_2,\dots)$.
	
	The function $f$ is injective, in the following manner. Suppose there exist distinct $(x,\omega),(x',\omega')\in U$ such that $f(x,\omega) = f(x',\omega') = (a_1,a_2,\dots)$. We consider the evolution of the front $\widetilde U$ as it reflects only off the obstacles $K_{a_1},K_{a_2},\dots$, in that exact order. Let $q(t) = \gamma(x,\omega)(t)$ and $q'(t) = \gamma(x',\omega')(t)$. Then applying \Cref{lemma:dist_on_front_reflection} we have \[d_X(q(0),q'(0))\leq d_{X_t}(q(t),q'(t))e^{-tk_{\min}}.\] 
	So $d_X(q(0),q'(0))\to 0$ as $t\to \infty$, that is $(x,\omega) = (x',\omega')$. Hence $f$ is injective, and therefore we consider $f$ as a bijection onto its image $f(U)\subseteq F$.
	
	Now consider a fixed $(x,\omega)\in U$, and let $(a_1,a_2,\dots)=f(x,\omega)$. For any $\varepsilon > 0$ there is an integer $N$ such that $2^{1-N}<\varepsilon$. Note that $(x,\omega)\in U$ guarantees $\gamma(x,\omega)$ is never tangent to an obstacle. Therefore there is a $\mu > 0$ such that the for any $(x',\omega')$ in the $\mu$-ball, $B_\mu$, around $(x,\omega)$ in $\widetilde U$, the ray $\gamma(x',\omega')$ transversally reflects off the obstacles $K_{a_1},K_{a_2},\dots,K_{a_N}$ in that order. Thus $f(x',\omega')$ and $f(x,\omega)$ must agree up to the $N$-th coordinate. Then for any $(x',\omega')\in B_\mu$ we have
	\[
		\eta (f(x',\omega'),f(x,\omega)) = \sum_{i=1}^\infty\frac{1}{2^i}(1-\delta_{f(x',\omega')_if(x,\omega)_i})\leq \sum_{i=N}^\infty \frac{1}{2^i} = \frac{1}{2^{N-1}}<\varepsilon.
	\]
	Hence $f$ is continuous. We shall complete the proof by showing that $f$ has a continuous inverse. Given $(x,\omega),(x',\omega')\in U$, let $(a_1,a_2,\dots) = f(x,\omega)$ and $(b_1,b_2,\dots) = f(x',\omega')$. Let $N$ be the largest integer such that $a_i = b_i$ for all $1\leq i \leq N$. Let $t_N>0$ be the time just after both $\gamma(x,\omega)$ and $\gamma(x',\omega')$ undergo the $N$-th reflection. Set $C = d_{X_{t_N}}(x_{t_N},x'_{t_N})$, where $x_{t_N} = \pi_1\circ\gamma(x,\omega)(t_n)$, and $x'_{t_N} = \pi_1\circ\gamma(x',\omega')(t_n)$. Then by \Cref{lemma:dist_on_front_reflection}, 
	\begin{equation}\label{eq:exp_dist_trapped_ray}
		d_X(x,x')< Ce^{-t_Nk_{\min}}.
	\end{equation}
	Now for any $\varepsilon>0$, pick $N$ so that $Ce^{-N\mindist k_{\min}}<\varepsilon$. Then for any $(x',\omega')$ such that $f(x,\omega)$ and $f(x',\omega')$ agree up to the $N$-th reflection, by \cref{eq:exp_dist_trapped_ray}, and since $t_N>N\mindist$, we must have $d_X(x,x')<\varepsilon$. Take any $0<\delta < 2^{-N}$, if $\eta(f(x,\omega),f(x',\omega'))<\delta$ it follows that $f(x,\omega)$ and $f(x',\omega')$ must agree up to at least the $N$-th reflection. Thus $\eta(f(x,\omega),f(x',\omega'))<\delta$ must imply $d_X(x,x')<\varepsilon$, i.e. the inverse of $f$ is also continuous. Therefore $f$ is a homeomorphism onto $f(U)$. Now since $f(U)\subseteq F$, it must have dimension 0, and hence $U$ must also have dimension 0 as desired.
\end{proof}

\begin{proof}[Proof of \Cref{Theorem 1}]
	Suppose that there exists some point $x_0\in \partial K$ such that $x_0\not\in \partial L$. Let $V\in T_{x_0}\partial K$ be any vector tangent to $\partial K$ at $x_0$. Then there is some neighbourhood $U\subseteq\partial K$ of $x_0$ such that $U\cap\partial L=\emptyset$. By \Cref{proposition:convex_front_from_obstacle} we may construct a convex front $X$ from $U$ (possibly shrinking it) such that the rays from $X$ in the inward normal direction intersect $U$ tangentially. Furthermore, let $\widetilde X$ be the set of points in $X$ for which the ray in $S_K$ in the normal direction is trapped, or tangent to $\partial K$ more than once. We also let $\widetilde X$ include any points in $x$ for which the corresponding ray in $S_L$ is tangent to $\partial L$ more than once. Then it follows from \Cref{proposition:tangent_twice} and \Cref{proposition:trapped_rays_dimension_zero} that $\dim(\widetilde X) \leq n-2$. Let $X' = X\backslash \widetilde X$, then $\dim(X')>0$. Given any $x\in X'$, note that the ray $\gamma^K_x$ which intersects $X$ at $x$ orthogonally, is tangent to $\partial K$ exactly once by construction. Then since $\gamma^K_x$ is not trapped, there is some $\sigma_x\in \TS$ for which $\gamma^K_x(t) = \pr_1\circ\genflow^K_t(\sigma_x)$. 
%	Since $\widetilde X$ also includes any points in $x$ for which the corresponding ray in $S_L$ is tangent to $\partial L$ more than once, once again by \Cref{proposition:tangent_twice}, we have $0 = \dim(\widetilde X) < \dim(X')$. 
	Now since $K$ and $L$ have the same set of travelling times, for each $x\in X'$ there is exactly one point $z(x)\in \partial L$ such that the ray $\gamma^L_x(t) = \pr_1\circ\genflow^L_t(\sigma_x)$ is tangent to $\partial L$ at $z(x)$. Denote the set of all such tangent points by $Z = \{z(x) : x\in X'\}$. Note that by construction $Z\cap\partial K=\emptyset$ and that $z:X'\to Z$ is a homeomorphism. 	
	 We claim that there exists some $z^*\in \partial L$ such that for all neighbourhoods $\Omega\subseteq\partial L$ of $z^*$ we have $\dim(\Omega\cap Z)>0$. Assume the contrary. Then for every $z\in\partial L$ there is a neighbourhood $\Omega_z\subseteq\partial L$ such that $\dim(\Omega_z\cap Z)=0$. Consider the open cover $\{\Omega_z\}_{z\in\partial L}$ of $\partial L$. Since $\partial L$ is compact, we may take a finite subcover $\{\Omega_{z_i}\}_{i=1}^\alpha$. It follows $Z$ is covered by the finite cover $\{\Omega_{z_i}\cap Z\}_{i=1}^\alpha$. Recalling that $\dim(\Omega_{z_i}\cap Z)=0$, this implies that $\dim(Z) = 0$, a contradiction.
	 
	 Let $t^*_K(x)$ and $t^*_L(x)$ be such that $\gamma^K_x(t^*_K(x))$ and $\gamma^L_x(t^*_L(x))$ are the points of tangency with $\partial K$ and $\partial L$ respectively. Suppose that $t^*_K(x_0)<t^*_L(x_0)$, then provided $U$ is sufficiently small, this hold true for all $x\in X$, i.e. $t^*_K(x)<t^*_L(x)$. Now since $K$ and $L$ are equivalent up to tangency, it follows that
	\[
		\gamma^K_x(t) = \gamma^L_x(t) \textrm{ for all } 0\leq t\leq t^*_K(x), \textrm{ and }x\in X.
	\]

	 We now construct a convex front $Y$ (via \Cref{proposition:convex_front_from_obstacle}) from a neighbourhood $\Omega\subseteq\partial L$ of $z^*$, such that the outward normal $N(y)$ of $Y$ points towards $X$. i.e. $\pr_1\circ\genflow_{t(y)}^{L}(y,N(y))\in X$ for some $t(y)>0$. Define the set $\mathcal{Q}$ of points $z(x)\in Z$ such that $\gamma_x^L$ intersects both $X$ and $Y$ orthogonally. Since $\gamma^K_x(t) = \gamma^L_x(t)$ for all $0\leq t\leq t^*_K(x)$, and $x\in X$, it follows that $\gamma_x^L$ will intersect $X$ always. But by construction, $\gamma_x^L$ will intersect $Y$ for all $z(x)\in\Omega\cap Z$. Thus  $\mathcal{Q} = \Omega\cap Z$. Therefore $\dim(\mathcal{Q})>0$. However \Cref{proposition:convex_front_collision} states that $\dim(\mathcal{Q}) = 0$, a contradiction. Note that we assumed that $t^*_K(x_0)<t^*_L(x_0)$. If we assume instead that $t^*_K(x_0)>t^*_L(x_0)$, then by letting $\mathcal{Q}$ be the set of points $z(x)\in Z$ such that $\gamma_x^K$ intersects both $X$ and $Y$ orthogonally, we reach the same contradiction. Thus our proof is complete.
	
\end{proof}

\begin{proposition}\label{lemma:gen_pos_equiv}
	Suppose that $K$ and $L$ have the same set of travelling times, and that both $K$ and $L$ are in general position, respectively. Then $(K,L)\in\mathcal {K}_T$, that is, $K$ and $L$ are equivalent up to tangency.
\end{proposition}
\begin{proof}
	We begin by noting a few facts. First, if $\gamma$ is a geodesic in $S_K$ that is tangent to $\partial K$, then the point of tangency is either the first or last point of contact between $\gamma$ and $\partial K$. This follows directly from the general position condition on $K$ (and $L$) as follows. Let $t_1,\dots,t_j$ be the times when $\gamma$ reflects off the obstacle. Suppose $\gamma$ had a tangency between the first and last reflection, at some time $t_i,\ 1<i<l$. Then $\gamma|_{[t_{i-1},[t_i+1]}$ is a smooth geodesic intersecting $K$ on three distinct components. Thus contradicting the general position condition on $K$. This clearly holds for $L$ as well. Second, note that if $\sigma\in\TS\backslash\Trap{\partial S}$ is such that geodesic $\gamma_\sigma^K$ generated by $\sigma$ in $S_K$ is tangent to $\partial K$, then there is a singularity in the travelling time function at the point corresponding to $\sigma$. Since $K$ and $L$ have the same travelling time function, it follows that $\sigma$ must also generate a geodesic $\gamma_\sigma^L$ in $S_L$ which is tangent to $\partial L$. Note that $\gamma_\sigma^L$ and $\gamma_\sigma^K$ are not necessarily the same geodesic.
	
	Now using the first fact, $\gamma_\sigma^K$ we may assume that $\gamma_\sigma^K$ is tangent to $\partial K$ at the first point of contact. If not we may simply reverse the direction of $\gamma_\sigma^K$. We claim that $\gamma_\sigma^L$ is also tangent to $\partial L$ at the first point of contact. Using the first and second facts, $\gamma_\sigma^L$ must be tangent to $\partial L$ at either the first or last reflection points (or both). Suppose, for contradiction that $\gamma_\sigma^L$ is tangent to $\partial L$ only at the last point of reflection. Let $\tau_K, \tau_L>0$ be the times of the first and last reflection points of $\gamma_\sigma^K$ and $\gamma_\sigma^L$ respectively. By \Cref{proposition:convex_front_from_obstacle} we may construct convex fronts $X$ and $Y$ about $\gamma_\sigma^K(\tau_K)$ and $\gamma_\sigma^L(\tau_L)$ in the directions $\dot\gamma_\sigma^K(\tau_K)$ and $\dot\gamma_\sigma^L(\tau_L)$ respectively. It now follows from the same argument as in the proof of \Cref{Theorem 1} that we reach a contradiction to \Cref{proposition:convex_front_collision}. Hence $\gamma_\sigma^K$ and $\gamma_\sigma^L$ must both be tangent at their first points of reflection. If $t_K^*,t_L^*>0$ are the times of tangency for $\gamma_\sigma^K$ and $\gamma_\sigma^L$ respectively, it follows that $\gamma_\sigma^K(t) = \gamma_\sigma^L(t)$ for all $0\leq t \leq \min\{t_K^*,t_L^*\}$. That is, general position implies that $K$ and $L$ are equivalent up to tangency.
\end{proof}

\bibliographystyle{elsarticle-num}

\begin{thebibliography}{}
\expandafter\ifx\csname url\endcsname\relax
  \def\url#1{\texttt{#1}}\fi
\expandafter\ifx\csname urlprefix\endcsname\relax\def\urlprefix{URL }\fi
\expandafter\ifx\csname href\endcsname\relax
  \def\href#1#2{#2} \def\path#1{#1}\fi

\end{thebibliography}


\begin{thebibliography}{10}

\bibitem{MR448262}
S.~Alexander.
\newblock Locally convex hypersurfaces of negatively curved spaces.
\newblock {\em Proc. Amer. Math. Soc.}, 64(2):321--325, 1977.

\bibitem{MR350662}
Richard~L. Bishop.
\newblock Infinitesimal convexity implies local convexity.
\newblock {\em Indiana Univ. Math. J.}, 24:169--172, 1974/75.

\bibitem{MR618545}
B.-Y. Chen and L.~Vanhecke.
\newblock Differential geometry of geodesic spheres.
\newblock {\em J. Reine Angew. Math.}, 325:28--67, 1981.

\bibitem{MR0482697}
Ryszard Engelking.
\newblock {\em Dimension theory}.
\newblock North-Holland Publishing Co., Amsterdam-Oxford-New York, 1978.

\bibitem{MR3600043}
Colin Guillarmou.
\newblock Lens rigidity for manifolds with hyperbolic trapped sets.
\newblock {\em J. Amer. Math. Soc.}, 30(2):561--599, 2017.

\bibitem{constructivealgorithm}
T.~Gurfinkel, L.~Noakes, and L.~Stoyanov.
\newblock Recovering obstacles from their traveling times.
\newblock {\em Chaos: An Interdisciplinary Journal of Nonlinear Science},
  32(12):123131, 2022.

\bibitem{MR3743701}
Guan Huang, Vadim Kaloshin, and Alfonso Sorrentino.
\newblock On the marked length spectrum of generic strictly convex billiard
  tables.
\newblock {\em Duke Math. J.}, 167(1):175--209, 2018.

\bibitem{MR949013}
Mitsuru Ikawa.
\newblock Decay of solutions of the wave equation in the exterior of several
  convex bodies.
\newblock {\em Ann. Inst. Fourier (Grenoble)}, 38(2):113--146, 1988.

\bibitem{LEERIEMANN}
John~M. Lee.
\newblock {\em Introduction to {R}iemannian manifolds}, volume 176 of {\em
  Graduate Texts in Mathematics}.
\newblock Springer, Cham, second edition, 2018.

\bibitem{MR962285}
Nadine~L. Menninga.
\newblock Immersions of positively curved manifolds into manifolds with
  curvature bounded above.
\newblock {\em Trans. Amer. Math. Soc.}, 318(2):809--821, 1990.

\bibitem{MR3359579}
Lyle Noakes and Luchezar Stoyanov.
\newblock Rigidity of scattering lengths and travelling times for disjoint
  unions of strictly convex bodies.
\newblock {\em Proc. Amer. Math. Soc.}, 143(9):3879--3893, 2015.

\bibitem{math9192434}
Lyle Noakes and Luchezar Stoyanov.
\newblock Convex obstacles from travelling times.
\newblock {\em Mathematics}, 9(19), 2021.

\bibitem{MR3952705}
Plamen Stefanov, Gunther Uhlmann, Andras Vasy, and Hanming Zhou.
\newblock Travel time tomography.
\newblock {\em Acta Math. Sin. (Engl. Ser.)}, 35(6):1085--1114, 2019.

\bibitem{MR869706}
Ivan~A. Tribuzy.
\newblock Convex immersions into positively-curved manifolds.
\newblock {\em Bol. Soc. Brasil. Mat.}, 17(1):21--39, 1986.

\bibitem{MR807598}
A.~Vetier.
\newblock Sinai billiard in potential field (construction of stable and
  unstable fibers).
\newblock In {\em Limit theorems in probability and statistics, {V}ol. {I},
  {II} ({V}eszpr{\'e}m, 1982)}, volume~36 of {\em Colloq. Math. Soc. J{\'a}nos
  Bolyai}, pages 1079--1146. North-Holland, Amsterdam, 1984.

\end{thebibliography}

\end{document}